\documentclass{amsart}
\usepackage{amsmath,amsthm}
\usepackage{amsfonts,amssymb}
\usepackage{enumerate}
\usepackage{accents,color}
\usepackage{graphicx}

\newcommand{\lvt}{\left|\kern-1.35pt\left|\kern-1.3pt\left|}
\newcommand{\rvt}{\right|\kern-1.3pt\right|\kern-1.35pt\right|}

\newtheorem{thm}{Theorem}

\newtheorem{lem}{Lemma}

\newtheorem{THEO}{Theorem}

\newtheorem{LEM}{Lemma}

\theoremstyle{remark}

\makeatletter
\newcommand{\bddots}{%
  \mathinner{\mkern1mu\raise\p@\vbox{\kern7\p@\hbox{.}}\mkern2mu
    \raise4\p@\hbox{.}\mkern2mu\raise7\p@\hbox{.}\mkern1mu}}
\makeatother
 \def \to {\rightarrow}

\def \N {\mathbb{N}}

\def \R {\mathbb{R}}
\def \C {\mathbb{C}}

\graphicspath{{./}}

\begin{document}

\title[Zeros of Dirichlet Polynomials via a Density Criterion]
{Zeros of Dirichlet Polynomials via a Density Criterion}

\author{Willian D. Oliveira}
\address{Departamento de Matem\'atica Aplicada\\
 IBILCE, Universidade Estadual Paulista\\
 15054-000 Sa\~{o} Jos\'e do Rio Preto, SP, Brazil.}
 \email{willian@ibilce.unesp.br}
 
 \thanks{Research supported by the Science Foundation of The State of S\~ao Paulo, Brazil, FAPESP, under Grants 2013/14881-9 and 2016/09906-0.}


\begin{abstract} 
We obtain a necessary and sufficient condition in order that a semi-plane of the form $\Re(s)>r$, $r\in \R$, is free of zeros of a given Dirichlet 
polynomial. The result may be considered a natural generalization of a well-known criterion for the truth of the Riemann hypothesis due to B\'aez-Duarte.
An analog for the case of Dirichlet polynomials of a result of Burnol which is closely related to B\'aez-Duarte's one is also established.   
\end{abstract}

\maketitle

\section{Introduction and statement of results}
\setcounter{equation}{0}

A series of recent results were devoted to the connection between zero-free regions of some Dirichlet series and density of certain spaces of functions in $L^2(0,\infty)$. 
The majority of the papers concern the Riemann zeta function because of its importance in the study of the distribution of primes. 

The space $\mathcal{B}$ of Beurling functions is defined by
\begin{equation}
\label{C}
\mathcal{B}=\left\{ h:(0,1)\to \C \ | \ h(x)=\sum_{k=1}^{n}b_k \left\{\frac{1}{\theta_kx}\right\},  \ b_k \in \C,  \ \theta_k \geq1, \  \sum b_k/\theta_k=0 \right\},
\end{equation}
where $\{x\}=x-[x]$ stands for the fractional part of $x$. If $\mathcal{B}^p$ denotes the closure of $\mathcal{B}$ in $L^p(0,1)$, the following classical result holds:
\begin{THEO}
\label{NB}
The Riemann zeta function $\zeta(s)$ does not vanish in the semi-plane  $\Re(s)>1/p$ if and only if  $\mathcal{B}^p=L^p (0, 1)$.
\end{THEO}
\noindent Nyman \cite{Nym} proved the result for $p=2$ in 1950 and Beurling \cite{Beu55} generalised it for $p>1$ in 1955. That is why 
nowadays the statement of the theorem is commonly known as the Nyman-Beurling criterion. 
Bercovici and Foias \cite{BerFoi} settled  the case $p=1$ in 1984. It is clear that  it does not hold for $p>2$. Since then various results related to Theorem \ref{NB} have been obtained \cite{B-D93, B-D05, BBLS, BS, BalRot, BetConFar, Bur, DFMR13, DKDWDO, Rot05, Rot07, Rot09, Nik}. For $p=2$ 
the above theorem provides a criterion for the truth of the Riemann hypothesis (RH). 
 
Given a real number $\lambda \geq1$, consider the subspace $\mathcal{B}_\lambda \subset \mathcal{B}$ which consists of the functions $f\in \mathcal{B}$, such that $\theta_k\leq  \lambda$ for every $k=1, \dots, n$, and denote by $D_{\lambda,p}$ the distance in  
$L^p(0,1)$ from the characteristic function $\textbf{1}_{(0,1)}$ of the interval $(0,1)$ to the the space $\mathcal{B}_\lambda$. Since Beurling \cite{Beu55} also showed that  $\mathcal{B}^p=L^p (0, 1)$ if and only if  $\textbf{1}_{(0,1)} \in \mathcal{B}^p$, then the above 
theorem can be reformulated as\\

\noindent \textbf{Theorem A}\textit{
The Riemann zeta function $\zeta(s)$ does not vanish in the semi-plane  $\Re(s)>1/p$ if and only if  
$$\lim_{\lambda \to \infty} D_{\lambda,p}=0.$$}

B\'aez-Duarte \cite{B-D03} showed that when $p=2$ the conditions $\theta_k\geq1$ in (\ref{C}) can be reduced to $\theta_k\in \N$ and, if we  substitute $L^2(0,1)$ by $L^2(0,\infty)$,
the restriction $\sum b_k/\theta_k=0$ can be removed. B\'aez-Duarte's refinements yield the following nice criterion for the Riemann hypothesis in terms of approximation of the 
characteristic function $\textbf{1}_{(0,1)}$: 
\begin{THEO}
\label{B}
The RH holds if and only if 
$\lim_{n \to \infty}d_n=0$, where
$$
d_n^2=\inf_{{b_1,\ldots,b_n\in \C}\atop}\int_0^\infty \left|\textbf{1}_{(0,1)}-\sum_{k=1}^n b_k \left\{\frac{1}{kx}\right\}\right|^2 dx.
$$
\end{THEO} 
\noindent The latter statement reduces the RH to an extremal problem about the best approximation of $\textbf{1}_{(0,1)}$ in a 
Hilbert space in terms of elements from a finite dimensional subspace and the solution of every such a problem is given by the projection. As it is clear from 
Theorem \ref{B}, $d_n$ is nothing but the distance in $L^2(0,\infty)$ from $\textbf{1}_{(0,1)}$ to the $n$-dimensional space $\mathrm{span} \{\rho_k(x):k=1,\ldots,n\}$, where $\rho_k(x)=\{1/kx\}$, 
so that
$$
d_n^2 =  \frac{\mathrm{det}\ G(\rho_1,\ldots ,\rho_n, \textbf{1}_{(0,1)})}{\mathrm{det}\ G(\rho_1,\ldots ,\rho_n)},
$$
where $G(\rho_1,\ldots ,\rho_n, \textbf{1}_{(0,1)})$ and $G(\rho_1,\ldots ,\rho_n)$ are the Gram matrices of the corresponding functions 
with respect to the usual inner product in $L^2(0,\infty)$.

It is know that if $d_n$ converges to zero, the speed of convergence would be relatively slow. More precisely, Burnol \cite{Bur}, generalizing an earlier result of B\'aez-Duarte, Balazard, Landreau and Saias \cite{BBLS}  proved that
\begin{THEO} 
\label{BurnolB}
The sequence $d_n$ satisfies
$$
\liminf_{n\rightarrow \infty}\ d_n^2 \log n \ \geq C,
$$ 
with
 $$
 C=\sum_{\Re(\rho)=1/2} \frac{m_\rho^2}{|\rho|^2},
$$
where the sum is over the distinct nontrivial zeros $\rho$ of the Riemann zeta function and $m_\rho$ are their multiplicities. 
\end{THEO}

The Nyman-Beurling criterion (Theorem \ref{NB}) was generalized recently for other general Dirichlet series, including Dirichlet polynomials (DP) \cite{Rot05,  DFMR13, DKDWDO}. In particular, given a  Dirichlet polynomial $P(s)=\sum_{k=1}^{m} a_k k^{-s}$ of order $m$, 
with $a_1\neq 0$, and $r\in \R$, define the function $\kappa_r(x):(0,\infty) \mapsto \C$ by 
$$
\kappa_r(x)=\sum_{k\leq x} \frac{a_k}{k^{r-1/2}}
$$
and the space $B_r$ by 
$$
\label{Ceta}
B_r := \left\{ h:(0,1) \mapsto \C\ | \ h(x)=\sum_{k=1}^{n} b_k\, \kappa_r \left(\frac{1}{\theta_kx}\right), \ b_k\in \C, \ \theta_k\geq 1 \right\}.
$$
By analogy with Theorem \ref{NB}, for every $\lambda \geq 1$, fix the subspace  $B_{\lambda,r}$ of $B_r$ to consist of those functions $f\in B_r$ that satisfy $\theta_k\leq  \lambda$ for every $k=1, \dots, n$, and  denote by $D_{\lambda,r}$ the distance in $L^2(0,1)$ from $\textbf{1}_{(0,1)}$ to $B_{\lambda,r}$. The results obtained in  \cite{Rot05,  DFMR13, DKDWDO} imply: 

\begin{THEO}
The Dirichlet polynomial $P(s)$ does not vanish for $\Re(s)>r$ if and only if
$$
\lim_{\lambda \to \infty} D_{\lambda,r}=0.
$$
\end{THEO}
Observe that the latter is the natural analog of the Nyman-Beurling criterion for DP. However, to the best of our knowledge, there are no results in the literature 
which may be considered analogs of Theorems B and C about DP. The main aim of this paper is to fill this gap. Our main results read as follows:

\begin{thm}\label{BP}
\label{BP} The Dirichlet polynomial  $P(s)=\sum_{k=1}^{m} a_k k^{-s}$, with $a_1\neq 0$,  does not vanish in the semi-plane  $\Re(s)>r$ if and only if $\lim_{n \to \infty}d_{n,r}=0$, where
$$
d_{n,r}^2=\inf_{{b_1,\ldots,b_n\in \C}\atop}\int_0^\infty \left|\textbf{1}_{(0,1)}-\sum_{k=1}^n b_k \kappa_r\left(\frac{1}{kx}\right)\right|^2 dx.
$$
\end{thm}
\noindent
and

\begin{thm} \label{BE}
For the Dirichlet polynomial  $P(s)=\sum_{k=1}^{m} a_k k^{-s}$, $a_1\neq 0$, the associated sequence $d_{n,r}$ satisfies
$$
\liminf_{n\rightarrow \infty}\ d_{n,r}^2 \log n \geq  C,
$$
with 
$$
C=\sum_{\Re(\rho_P)=r}  \frac{1}{|\rho_P-r+1/2|^2},
$$ 
where the sum is over the zeros $\rho_P$ of $P(s)$ which belong to the line $\Re(s)=r$. 
If $P$ does not possess zeros on $\Re(s)=r$ the constant $C$ is defined to be zero. 
\end{thm}

It is worth mentioning that de Roton \cite{Rot05, Rot06} obtained nice generalizations of Theorems B and C  for zeros of $L$-functions in the Selberg class. 
The methods employed in \cite{Rot05, Rot06} require deep analysis and a crucial fact which allows the generalizations is that those functions obey a functional equation which is a natural analog of the 
corresponding one for the Riemann zeta function itself. However, in general the Dirichlet polynomials do not obey such structural properties, so that the proofs of our main results require 
different techniques. Another feature of Theorems 1 is that our criterion concerns lack of zeros in the semi-planes $\Re(s)>r$, for every $r\in \R$, while both Baez-Duarte's criterion and 
de Roton's generalization deal exclusively with the semi-plane $\Re(s)>1/2$. From this point of view Theorem 1 in the present paper is more similar to the result obtained in \cite{DKDWDO} where 
Baez-Duarte's result was extended for a class of $L$-functions which is a bit narrower than in de Roton's result but for zeros in the semi-plane of the form $\Re(s)>1/p$ for any $p\in (1,2]$.

\section{Preliminary results}
\setcounter{equation}{0}

In this section we furnish the necessary definitions and known results for the proofs of the main theorems.
\subsection{Definitons and basic properties of DP}

Let  $P(s)=\sum_{k=1}^{m} a_k k^{-s}$ be a DP of order $m$ with $a_1\neq 0$. Recall that we associate with $P$ the function  
$\kappa_r(x):\R_+ \mapsto \C$, where $R_+=(0,\infty)$, defined by  
\begin{equation}\label{capa}
\kappa_r(x)=\sum_{k\leq x} \frac{a_k}{k^{r-1/2}}
\end{equation}
and the space 
\begin{equation}
\label{Ceta}
C_r := \left\{ h:(0,1) \mapsto \C\, | \, h(x)=\sum_{k=1}^{n} b_k\, \kappa_r \left(\frac{1}{kx}\right), \ b_k\in \C \right\}.
\end{equation}

The inequality  
$$
|P(s)-a_1|\leq \sum_{k=2}^m \frac{|a_k|}{k^{\Re(s)}}
$$
shows that $P(s)\to a_1$ uniformly with respect to  $\Im(t)$ when $\Re(s)\to \infty$.  On the other hand, $P(s)m^s\to a_m$ also uniformly with respect to  $\Im(t)$ when $\Re(s)\to -\infty$. 
Therefore, since $a_1$ and $a_m$ are nonzero, there are real constants $\alpha$ and $\beta$ associated with $P$ such that 
\begin{equation}\label{00}
\alpha \leq \Re(\rho_P)\leq \beta
\end{equation}
for every zero  $\rho_P$ of $P$. 

Since $P(s)$ is a finite sum of exponential functions, it is an entire function of order one and finite type. Then, by the Hadamard factorization theorem, there are complex constants $a, b$ and a natural number $q$ such that 
$$
P(s)=s^q e^{a+bs} \prod_{\rho_P}\left( 1-\frac{s}{\rho_P}\right)e^{s/\rho_P},
$$
where $\rho_P$ runs over the zeros of $P$ arranged in increasing order of $|\rho_P|$ and $\sum_{\rho_P}|\rho_P|^{-2}<\infty$. Moreover the infinite product converges uniformly on the compacts sets of $\C$.

Given a Dirichlet series $f(s)=\sum_{k=1}^{\infty}a_kk^{-s}$, with $a_1\neq 0$, we denote by $\sigma_a=\sigma_a(f)$ its abscissa of absolute convergence and we set  
$\sum_{k=1}^{\infty}\mu_f(k)k^{-s}$ to be the formal Dirichlet series associated with $1/f(s)$. 

\subsection{The Mellin Transform}
The Mellin transform of a function $f:\R_+\mapsto \C$ is the complex-valued function 
\begin{equation}\label{MT}
\mathcal{M}\left[f(x);s\right]=\mathcal{M}f(s)=\frac{1}{\sqrt{2\pi}}\int_0^\infty f(x)x^{s-1}dx,
\end{equation}
defined for those complex numbers $s$ for which the integral converges and the Fourier transform $\mathcal{F}$ of a function  $f\in L^1(\R)$ is defined by 
$$
\mathcal{F}\left[f(x);t\right]=\mathcal{F}f(t)=\frac{1}{\sqrt{2\pi}}\int_{-\infty}^\infty f(x)e^{-i x t}dx.
$$
Note that the terms Mellin transform and Fourier transform are also applied to the mapping which takes $f$ to $\mathcal{M}f$ and $f$ to $\mathcal{F}f$. As it is known,
the Plancherel theorem allows the Fourier transform to be extended to an isometry $\mathcal{F}:  L^2 (\R) \mapsto L^2(\R)$, called the extended Fourier transform or the Fourier-Plancherel transform. Consider the isometry $\mathcal{I}:L^2(\R_+) \mapsto  L^2(\R)$ which associates, with each $f \in L^2(\R_+)$, the function $g(u)=f(e^{-u})e^{-u/2}$ in $L^2 (\R)$. Then the extended Fourier transform allows us to define the extended Mellin transform $\mathcal{M}:  L^2(\R_+) \mapsto  L^2(\Re(s)=1/2)$ by
$$
\mathcal{M}\left[f(x);s\right]=\mathcal{F}\left[ \mathcal{I}f(x);t\right], \ \ \ s=1/2+it.
$$
When $\mathcal{I}f \in L^1(\mathbb{R})$, for some $f\in L^2(\R_+)$, the Mellin transform $\mathcal{M}\left[f(x);1/2+it\right]$ is given by (\ref{MT}).

In the proof of  Theorem \ref{BP} we shall need the Mellin transform of the functions in $C_r$. It follows from Abel's identity that 
$$
\sum_{k\leq n} \frac{a_k}{k^{r-1/2}} \frac{1}{k^s}=\frac{\kappa_r(n)}{n^s} +s\int_1^n \frac{\kappa_r(y)}{y^{s+1}}dy.
$$
Since $\kappa_r$ is bounded in $\R_+$, if $\Re(s)>0$, we let $n \to \infty$ in the latter to obtain
$$
\frac{P(s+r-1/2)}{s}= \int_1^\infty \frac{\kappa_r(y)}{y^{s+1}}dy.
$$
If $k \in \N$ the change of variables $y=1/(kx)$ yields 
$$
\frac{P(s+r-1/2)}{s}k^{-s}= \int_0^{1/k} \kappa_r(1/kx)\, x^{s-1}\, dx=  \int_0^{1} \kappa_r(1/kx)\, x^{s-1}\, dx.
$$
Therefore, for every $h(x)=\sum_{k=1}^{n} b_k\, \kappa_r (1/kx)$ in $C_r$, we obtain
\begin{equation}
\label{Int1f}
\int_0^1 h(x)x^{s-1}dx=\frac{P(s+r-1/2)}{s}\sum_{k=1}^{n} \frac{b_k}{k^{s}}, \ \  \Re(s)>0.
\end{equation}

\subsection{Some technical results} 
 We shall need the first effective Perron formula  (see \cite[Chapter II.2, p. 132]{Ten}):
 \begin{LEM}
\label{LemTit}
 Let $F(w)=\sum a_k k^{-w}$ be a Dirichlet series. Then, for every $c> \max \{0,\sigma_a\}$, $T\geq 1$, and $x\geq1$ not an integer, we have
 $$
\sum_{k<x} a_k  =  \frac{1}{2 \pi i} \int_{c-iT}^{c+iT}F(w)\frac{x^w}{w}dw+O\left( x^c \sum_{k=1}^\infty \frac{|a_k|}{k^c (1+T|\log(x/k)|)}\right).
$$
\end{LEM}
\noindent The following is Lemma 6.30 in \cite{BatDia}:

\begin{LEM}
\label{LemBat}
 Let $f(s)=\sum a_k k^{-s}$ be a Dirichlet series with $a_1\neq 0$. The Dirichlet series $f(s)$ has finite abscissa of convergence if and only if $f^{-1}(s)=\sum \mu_f(k) k^{-s}$ has finite abscissa of convergence. 
 \end{LEM}

Next we prove two technical lemmas. 
\begin{lem}
\label{LemLit}
 Let $f(s)=\sum a_k k^{-s}$ be a Dirichlet series with $a_1\neq 0$ and abscissa of absolute convergente $\sigma_a$. If $f(s)$ does not vanish in $\Re(s)>\sigma_z\geq \sigma_a$ then for each $\epsilon>0$
 $$
 |f(s)|^{\pm1}\leq C_\epsilon,
 $$
 uniformly in $\Re(s)\geq \sigma_z+\epsilon$.
\end{lem}

\begin{proof}  Let  $b>\sigma_z+\epsilon$.  Since the Dirichlet series $f(s)$ converges absolutely for $\Re(s)>\sigma_z$, there exists a positive constant $C_1(\epsilon)$, such that $\Re(\log f(s))=\log |f(s)|<C_1(\epsilon)$ for $\Re(s)\geq \sigma_z+\epsilon/2$. Next we apply the Borel-Carath\'eodory theorem for $\log f(s)$, which is holomorphic for $\sigma>\sigma_z$, and  the concentric circumferences with centre at $b+i t$ and radii $b-\sigma_z-\epsilon/2$ and $ b-\sigma_z-\epsilon$.  Then in the smaller circumference 
\begin{equation}\label{bci}
|\log f(s)|  \leq  \frac{4b-4\sigma_z-4\epsilon}{\epsilon}C_1(\epsilon)+\frac{4b-4\sigma_z-3\epsilon}{\epsilon} |\log f(b+it)|.
\end{equation}
Observe that the straightforward inequality
$$
|f(b+it)-a_1|\leq \sum_{k=2}^\infty \frac{|a_k|}{k^{b}}\leq \frac{1}{2^{b-\sigma_z-\epsilon}} \sum_{k=2}^\infty \frac{|a_k|}{k^{\sigma_z+\epsilon}}
$$
and the fact that $a_1\neq 0$ implies that for a fixed  sufficiently large $b$ there exists another positive constant $C_2(\epsilon)$ such that $|\log f(b+it)|<C_2(\epsilon)$ for every real $t$. Setting the latter inequality into (\ref{bci}) we conclude that 
$$
|\log f(s)|  \leq (4b-4\sigma_z) \frac{\max\{C_1(\epsilon),C_2(\epsilon)\}}{\epsilon}=C(\epsilon,b), \ \ \ \mathrm{for} \ \ \ \sigma_z+\epsilon \leq \Re(s)\leq b.
$$
Hence
$$
 |f(s)|^{\pm1}\leq e^{C(\epsilon,b)}, \ \  \ \sigma_z+\epsilon \leq \Re(s)\leq b.
 $$
Applying the Phragm\'en-Lindel\"of convexity principle, as stated in  \cite[Theorem 5.53]{IwaKow}, we obtain the desired result.
\end{proof}

\begin{lem}
\label{LemLan}
 Let $f(s)=\sum a_k k^{-s}$ be a Dirichlet series with $a_1\neq 0$ and abscissa of absolute convergente $\sigma_a$. If $f(s)$ does not vanish in $\Re(s)> \sigma_z\geq \sigma_a$ then, 
 for each $\epsilon>0$ and every $\delta \in (0,\epsilon)$ we have 
 $$
\sum_{k=1}^n\frac{ \mu_f(k)}{k^s}=\frac{1}{f(s)}+O_{\epsilon,\delta}(n^{-\delta})
 $$
 uniformly in the half-plane $\Re(s)\geq \sigma_z+\epsilon$. Moreover,
 $$
 \sum_{k=1}^n \mu_f(k)k^{-s}=O_{\epsilon}(1)
 $$
uniformly in $n=1, 2, 3,\dots$  and uniformly in the half-plane $\Re(s)\geq \sigma_z+\epsilon$.
 \end{lem}

\begin{proof} For $s$ in the half-plane $\Re(s)\geq \sigma_z+\epsilon$, consider the function
$$
F(w)=\frac{1}{f(s+w)}=\sum_{k=1}^\infty \frac{\mu_f(k)}{k^s} \frac{1}{k^w}.
$$
By Lemma \ref{LemBat} the Dirichlet series $F(w)$ has a finite abscissa of absolute convergence $\sigma_a(F)$. Let us apply Lemma \ref{LemTit} for fixed $c>\max \{0, \sigma_a(F)+1\}$ and $x\in 1/2+ \N$ to obtain
$$
\sum_{k<x} \frac{ \mu_f(k)}{k^s}=  \frac{1}{2 \pi i} \int_{c-iT}^{c+iT}F(w)\frac{x^w}{w}dw+O\left( x^c \sum_{k=1}^\infty \frac{|\mu_f(k)|}{k^{\Re(s)+c} (1+T|\log(x/k)|)}\right).
$$
Observe that, for each $k\in \N$, we have $|\log(x/k)|\geq |\log(1+1/2k)|\geq \log 2/(2k)$. Thus
$$
O\left( x^c \sum_{k=1}^\infty \frac{|\mu_f(k)|}{k^{\Re(s)+c} (1+T|\log(x/k)|)}\right)=O\left( \frac{x^c}{T} \sum_{k=1}^\infty \frac{|\mu_f(k)|}{k^{\Re(s)+c-1} }\right)=O\left( \frac{x^c}{T}\right).
$$
Deforming the contour of integration along the segment $[c-iT,c+iT]$ into the polygonal line passing through the points $-\delta-iT$ and $-\delta+iT$ and having in mind that the value of the residue of the integrand at the pole $w=0$ is $1/f(s)$, 
we obtain
$$
\sum_{k<x} \frac{ \mu_f(k)}{k^s}=\frac{1}{f(s)} + \frac{1}{2 \pi i}\left( \int_{-\delta+iT}^{c+iT}+\int_{-\delta-iT}^{-\delta+iT}+ \int_{c-iT}^{-\delta-iT}\right) F(w)\frac{x^w}{w}dw+O\left( \frac{x^c}{T}\right).
$$
By Lemma \ref{LemLit} the first and the third integrals can be estimated from above by 
$$
O_{\epsilon,\delta}\left(\frac{x^c}{T}\right)
$$
and the second one by 
$$
O_{\epsilon,\delta}\left( x^{-\delta} \right).
$$
Choosing $T=x^{c+\delta}$ we obtain, for each $\epsilon>0$, every $\delta \in (0,\epsilon)$ and $s$ in the half plane $\Re(s)\geq \sigma_z+\epsilon$,
$$
\sum_{k<x} \frac{ \mu_f(k)}{k^s}=\frac{1}{f(s)} +O_{\epsilon,\delta}\left( x^{-\delta} \right).
$$
In particular 
$$ 
\sum_{k=1}^n\frac{ \mu_f(k)}{k^s}=\frac{1}{f(s)} +O_{\epsilon,\delta}\left( n^{-\delta} \right)
$$
uniformly in  $\Re(s)\geq \sigma_z+\epsilon$. Choosing $\delta=\epsilon/2$,  by Lemma \ref{LemLit} we obtain
 $$
 \sum_{k=1}^n \mu_f(k)k^{-s}=\frac{1}{f(s)}+O_{\epsilon}\left(n^{-\epsilon/2}  \right)=O_{\epsilon}(1)
 $$
uniformly in $n=1, 2, 3,\dots$  and uniformly in the half plane $\Re(s)\geq \sigma_z+\epsilon$.

\end{proof}

\subsection{Lubinsky's Dirichlet orthogonal polynomials}
\label{LubPol}
For any strictly increasing sequence of real numbers $1= \lambda_1 < \lambda_2 < \lambda_3 < \cdots$, Lubinsky \cite{LubJAT} considered the general Dirichlet polynomials formed by linear combinations of $\lambda_k^{-it}$
and obtained the corresponding orthogonal basis with respect to the arctangent density.
He proved that  the general Dirichlet polynomials 
$\phi_1(t)=1$, $\phi_n(t)=(\lambda_n^{1-it} - \lambda_{n-1}^{1-it})/\sqrt{\lambda_n^{2} - \lambda_{n-1}^{2}}$, $n\geq 2$, satisfy 
$$
\int_\mathbb{R} \phi_n(t) \overline{\phi_m(t)} \frac{dt}{\pi (1+t^2)} = \delta_{nm},\ \ n,m \in \mathbb{N}.
$$ 
Choosing $\lambda_n=n^{1/2}$ and performing a simple change of variables one concludes that the Dirichlet polynomials
$\psi_{n}(t) =  n^{1/2-it} - (n-1)^{1/2-it}$, $n\geq 2$, obey
$$
\frac{1}{2 \pi} \int_\mathbb{R} \psi_{n}(t) \overline{\psi_{m}(t)} \frac{dt}{1/4+t^2} = \delta_{nm},\ \ n,m \in \mathbb{N}.
$$ 
Moreover Lubinsky \cite[(1.20), (1.19)]{LubJAT} proved that the kernel polynomials
$$
K_{n}(u,v) = \sum_{k=1}^n \psi_{k}(u) \overline{\psi_{k}(v)}
$$ 
satisfy the following uniform asymptotic estimates in compact subsets of $\mathbb{R}$, as $n \rightarrow \infty$:
\begin{equation}
\label{LubEst1}
K_{n}(u,u) =  |1/2+i u|^2\, \log n\, (1+o(1))
\end{equation}
and 
\begin{equation}
\label{LubEst2}
|K_{n}(u,v)| \leq \frac{2|1/2+i u|\, |1/2-i v|}{|u-v|} + o(\log n).
\end{equation}
The following simple fact holds: 
\begin{lem}
\label{InvMat}
For every $l\in \mathbb{N}$ and for any distinct numbers $t_1,\ldots, t_l \in \mathbb{R}$, there exists $n(l)\in \mathbb{N}$ such that the self-adjoint matrix 
$H = \left( K_{n}(t_i,t_j) \right)_{i,j=1}^{\ l}$ 
is nonsingular for every $n>n(l)$. Moreover,
\begin{equation}
\label{Red}
\frac{\det H}{(\log n)^l}=  |1/2+i t_1|^2\dots |1/2+i t_l|^2+o(1)\ \ \mathrm{as}\ \ n \rightarrow \infty.
\end{equation}
\end{lem}

\begin{proof} The Leibniz formula for the expansion of the determinant $H$ in terms of the permutations $\mathcal{P}_l$ and the above asymptotic formulae for the kernel 
polynomials yield  
\begin{eqnarray*}
\det H & = & \sum_{\sigma \in \mathcal{P}_l}  \mathrm{sgn} (\sigma) K_{n,r}(t_1,t_{\sigma(1)})\dots K_{n,r}(t_l,t_{\sigma(l)})\\
& = & |1/2+i t_1|^2\dots |1/2+i t_l|^2 (\log n)^l(1+o(1))+O((\log n)^{l-2})\\
& = & |1/2+i t_1|^2\dots |1/2+i t_l|^2 (\log n)^l+o((\log n)^l),
\end{eqnarray*}
which is equivalent to (\ref{Red}). Thus, obviously $H$ is nonsingular for all sufficiently large $n$.  
\end{proof}

\section{The analog of B\'aez-Duarte's criterion for Dirichlet Polynomials}

%\begin{thm}
%The polynomial  $P(s)=\sum_{k=1}^{m} a_k k^{-s}$, with $a_1\neq 0$,  does not vanish in the semi-plane  $\Re(s)>r$ if and only if $\lim_{n \to \infty}d_{n,r}=0$, where
%$$
%d_{n,r}^2=\inf_{{b_1,\ldots,b_n\in \C}\atop}\int_0^\infty \left|\textbf{1}_{(0,1)}-\sum_{k=1}^n b_k \kappa_r\left(\frac{1}{kx}\right)\right|^2 dx.
%$$
%\end{thm}

\textit{Proof of Theorem} 1. Suppose first that $\lim_{n\to \infty} d_{n,r}=0$. This means that $\textbf{1}_{(0,1)}$ belongs to the closure of the space 
$C_r$, defined by (\ref{Ceta}), in $L^2(0,1)$. Then, given $\epsilon>0$, there exists $h \in C_r$ such that $\|\textbf{1}_{(0,1)} - h \|_2<\epsilon$. By (\ref{Int1f}) 
$$
\int_0^1 (1-h(x))x^{s-1}dx=\frac{ 1-P(s+r-1/2)\sum_{k=1}^{n} b_k k^{-s}} {s}, \ \ \  \Re(s)>0.
$$
It is clear that  $x^{s-1}\in L_2(0,1)$ provided $\Re(s)>1/2$. Moreover,
$$
\| x^{s-1} \|_2^2 = \frac{1}{2\Re(s)-1}.
$$
H\"older's inequality yields 
$$
|1-P(s+r-1/2)\sum_{k=1}^{n}b_k k^{-s}|^2 < \epsilon^2 \frac{|s|^2}{2\Re(s)-1}.
$$
Assume that $P(s)$ possesses a zero $\rho_P$ in the semi-plane $\Re(s)>r$. Letting $\epsilon \to 0$ in the latter inequality we obtain an obvious contradiction. Therefore $P(s)$ does not vanish in  $\Re(s)>r$.

Now suppose that $P(s)$ does not vanish in $\Re(s)>r$. In order to show that $d_{n,r}$ converges to zero  it suffices to show that the function $\textbf{1}_{(0,1)}$  belongs to the closure of $C_r$ in $L^2(0,1)$. However, the latter statement is obviously equivalent  to the fact that $\mathcal{M}(\textbf{1}_{(0,1)})$ belongs to the closure of $\mathcal{M}(C_r)$ in  $L^2(\Re(s)=1/2)$ because the Mellin transform $\mathcal{M}:  L^2(\R_+) \mapsto  L^2(\Re(s)=1/2)$ is an isometry. Thus our aim is to prove that  the closure of $\mathcal{M}(C_r)$ contains  the function $\mathcal{M}(\textbf{1}_{(0,1)})$. 

From now on, up to the end of the present theorem, we set $s=1/2+it$. First we observe that, by (\ref{Int1f}), 
$$
\mathcal{M}(C_r)=\left\{\frac{ P(s+r-1/2)\sum_{k=1}^{n} b_k k^{-s}}{\sqrt{2\pi}s} \, : b_k \in \C \right\}.
$$
For each $n\in \N$ and $\epsilon>0$ we define the function 
$H_{n,\epsilon} \in L^2(\Re(s)=1/2)$ by
$$
H_{n,\epsilon} (s)=\frac{P(s+r-1/2)}{{\sqrt{2\pi}s}}\sum_{k=1}^n \frac{\mu_P(k)}{k^{r+\epsilon-1/2}}\frac{1}{k^s},
$$
where $\mu_{P}(k)$ are the coefficients of the expansion of $1/P$ in a formal Dirichlet series. Clearly  $H_{n,\epsilon}$ belongs to $\mathcal{M}(C_r)$. Since by hypothesis the Dirichlet Polynomial $P(s)$ does not vanish when $\Re(s)>r$ and converges absolutely in this semi-plane, for every fixed $\epsilon$ we can apply Lemma \ref{LemLan} to obtain 
$$
\lim_{n\to \infty} H_{n,\epsilon}(s)=H_{\epsilon}(s),
$$
where
$$
H_{\epsilon}(s):=\frac{1}{\sqrt{2\pi}s}\, \frac{P(s+r-1/2)}{P(s+r+\epsilon-1/2)}.
$$

It follows also from Lemma \ref{LemLan}  that  for every fixed $\epsilon$ the modulus of $H_{n,\epsilon}$ is uniformly bounded with respect to $n$ by a function from $L^2(\Re(s)=1/2)$.
Hence, by Lebesgue's dominated convergence theorem, for every fixed $\epsilon>0$, 
$$\lim_{n\to \infty} H_{n,\epsilon}=H_\epsilon,$$
in the $L^2(\Re(s)=1/2)$ norm. Since $H_{n,\epsilon} \in \mathcal{M}(C_r)$ for every fixed $\epsilon>0$, then $H_{\epsilon}$ belongs to the closure of  $\mathcal{M}(C_r)$. 

By the definition of $H_\epsilon$,
$$
\lim_{\epsilon \to 0} H_\epsilon(s)=\frac{1}{\sqrt{2\pi}s}=\mathcal{M}[\textbf{1}_{(0,1)}(x);s].
$$
Next we  show that the modulus of $H_\epsilon$ is uniformly bounded with respect to $\epsilon$, as $\epsilon \to 0$,  by a function from $L^2(\Re(s)=1/2)$. 
Recall the Hadamard factorisation of $P(z)$, 
$$
P(z)=z^q e^{a+bz} \prod_{\rho_P}\left( 1-\frac{z}{\rho_P}\right)e^{z/\rho_P},
$$
where $\sum_{\rho_P}|\rho_P|^{-2}$ converges and that the infinite product is uniformly convergent on compacts sets of $\C$. The hypothesis that $P(z)$ does not vanish in $\Re(z)>r$ implies that
$$
\left| \frac{P(r+it)}{P(r+\epsilon+it)}\right|\leq \left| \frac{r+it}{r+\epsilon+it}\right|^q \exp(\epsilon|\Re(b)|+\epsilon |\inf_{\rho_P} \Re(\rho_P)| \sum_{\rho_P}|\rho_P|^{-2}).
$$
Next we restrict $\epsilon \in (0,1)$ if $r\geq 0$ and  $\epsilon \in (0,|r|/2)$ otherwise and recall that $|\inf_{\rho_P} \Re(\rho_P)| <\infty$ holds for any Dirichlet polynomial $P(z)$ with zeros $\rho_P$(see (\ref{00}). Thus, there exist a positive constant $C$ such that  
$$
\left| \frac{P(r+it)}{P(r+\epsilon+it)}\right|\leq C
$$
when $\epsilon \to 0$. This estimate yields the desired upper bound
$$
| H_\epsilon(s) | \leq \frac{C}{|s|},
$$
which obviously belongs to $L^2(\Re(s)=1/2)$. Hence, by Lebesgue's theorem $\mathcal{M}(\textbf{1}_{(0,1)})$ belongs to the closure of  $\mathcal{M}(C_r)$ in the $L^2(\Re(s)=1/2)$ norm and this completes the proof.

\section{The analog of Burnol's lower bound for Dirichlet Polynomials}
%\begin{thm} Seja  $P(s)=\sum_{k=1}^{m} a_k k^{-s}$ um polinômio de Dirichlet de grau $m$ e  $a_1\neq 0$.  Então
%$$
%\liminf_{n\rightarrow \infty}\ d_n^2(r) \log n \geq  C,
%$$
%com 
%$$
%C=\sum_{\Re(\rho_P)=r} \frac{1}{|\rho_P-r+1/2|^2}
%$$ 
%e $\rho_P$ variando  no conjunto de zeros de $P(s)$ em $\Re(s)=r$. Caso o polinomio nao se anule na reta  $\Re(s)=r$ a constante $C$ deve ser definida com valor zero. 
%\end{thm}
{\em Proof of Theorem} \ref{BE}.  If the polynomial $P(s)$ does not have zeros on the line $\Re(s)=r$ the constant $C$ must be zero, as in the statement of the theorem. 

Suppose now that $P(s)$ possesses $l$ distinct zeros $r+it_1$, $\dots$, $r+it_l$ on the line $\Re(s)=r$. By (\ref{Int1f}) and the fact that the Mellin transform is an isometry 
$$
d_{n,r}^2=\inf_{{A_n\in \mathcal{D}_n}\atop} \frac{1}{2\pi} \int_{\Re(s)=1/2}\left| \frac{1-P(s+r-1/2)A_n(s)}{s} \right|^2 |ds|,
$$
where $\mathcal{D}_n$ is the space of Dirichlet polynomials of the form $\sum_{k=1}^n b_k k^{-s}$. Let us consider $Q_r(s)=P(s+r-1/2)$ and the subspace $\mathcal{D}_{n,l} \subset \mathcal{D}_n$ composed by all Dirichlet polynomials  $B_n(s)$ which satisfy the $l$ interpolation conditions $B_n(1/2+it_1)=1$, $\dots$, $B_n(1/2+it_l)=1$. Since $1-Q_r(s)A_n(s)$ is a DP of order $mn$ 
and obeys the same interpolations property, then
$$
d_{n,r}^2\geq \inf_{{B\in \mathcal{D}_{nm,l}}\atop} \frac{1}{2\pi} \int_{\Re(s)=1/2}\left| \frac{B(s)}{s} \right|^2 |ds|.
$$
Observe that $\mathcal{D}_{nm,l}$ is nonempty because  it contains $B(s)\equiv1$.  If $B_{nm} \in \mathcal{D}_{nm,l}$ then $B_{nm}(1/2+it_j)=1$, $j=1,\dots,l$ and $B_{nm}(1/2+it)=\sum_{k=1}^{nm}b_k\psi_k(t)$, where $\psi_k$ are defined in Section \ref{LubPol}. The  interpolation conditions can be rewritten in the form  
$$
A_{l,nm}\textbf{B}=\mathbf{1}_l, 
$$
where
$$
A = A_{l,nm}=\left(\begin{array}{cccc}

\psi_1(t_1) & \psi_2(t_1) & \dots & \psi_{nm}(t_1)\\

\psi_1(t_2) & \psi_2(t_2) & \dots & \psi_{nm}(t_2)\\

\vdots                    &  \vdots                   & & \vdots \\

\psi_1(t_l) & \psi_2(t_l) & \dots & \psi_{nm}(t_l)

\end{array}\right),
$$
$\textbf{B}=(b_1,\ldots, b_{nm})^T$ and $\mathbf{1}_l$ is the column vector of size $l$ whose entries are all equal to one.
Then obviously 
$$
 \|B_{nm}(1/2+it) \|_{L^2(\R,\omega)}^2= |b_1|^2+|b_2|^2+\dots+|b_{nm}|^2,  
$$
where $L^2(\R,\omega)$ is the weighted $L^2$ space with weight $\omega(t)= (1/2\pi)(1/4+t^2)$. Thus the problem reduces to minimize $\|\textbf{B}\|^2$, $\textbf{B}\in \mathbb{C}^{nm}$, subject to
$A_{l,nm}\textbf{B}=\mathbf{1}_l$. The solution of this problem is known to be given by the solution of the following linear system (see \cite[Theorem 2.19]{Bre} and \cite{GreRos, Nev}): 
\begin{eqnarray*}
A\,\textbf{B} & = & \mathbf{1}_l,\\
\textbf{B} & = & A^\ast \lambda,\ \ \lambda \in \mathbb{C}^l.
\end{eqnarray*}
By Lemma \ref{InvMat} there exists $n(l)\in \mathbb{N}$ such that the self-adjoint matrix $H = A A^*= \left( K_{nm}(t_i,t_j) \right)_{i,j=1}^{\ l}$ is nonsingular for every $n>n(l)$. Hence, for $n>n(l)$, the above system 
has the unique solution  $\tilde{\textbf{B}} = A^\ast H^{-1} \mathbf{1}_l$. Since  $A\, A^\ast = H$ and $H$ is self-adjoint, then
$$
 \|\tilde{B}_{nm}(1/2+it) \|_{L^2(\R,\omega)}^2=   \tilde{\textbf{B}}^\ast \tilde{\textbf{B}} =  \mathbf{1}_l^\ast (H^{-1})^\ast  A\, A^\ast H^{-1} \mathbf{1}_l = \mathbf{1}_l^\ast H^{-1} \mathbf{1}_l. 
$$
Thus, by the Cramer formula for the inverse matrix
$$
  \|\tilde{B}_{nm}(1/2+it) \|_{L^2(\R,\omega)}^2 = \sum_{i,j=1}^l (-1)^{i+j} \frac{\det H_{ij}}{\det H},
$$
where $H_{ij}$ are the $(i,j)$-th cofactors of $H$. 
Moreover, the asymptotic relations (\ref{LubEst1}) and (\ref{LubEst2}) and Lemma  \ref{InvMat}  yield the following ones, as $n \rightarrow \infty$:
\begin{eqnarray*}
\det H_{jj} & \sim & (\log nm)^{l-1}\, \frac{|1/2+i t_1|^2\dots |1/2+i t_l|^2}{|1/2+i t_j|^2}\\
\det H_{ij} & = & O((\log nm)^{l-2}),\ \ i\neq j\\
\det H & \sim &  (\log nm)^l  |1/2+i t_1|^2\dots |1/2+i t_l|^2.
\end{eqnarray*}
Hence, 
$$
\|\tilde{B}_{nm}(1/2+it) \|_{L^2(\R,\omega)}^2 \sim \frac{1}{\log n+\log m} \sum_{j=1}^l  \frac{1}{1/4+t_j^2},\ \ \mathrm{as}\ \ n\rightarrow \infty,
$$
Since $d^2_{n,r}\geq  \|\tilde{B}_{nm}(1/2+it) \|_{L^2(\R,\omega)}^2$, we obtain
$$
\liminf_{n\to \infty} d^2_{n,r} \log n \geq  \sum_{j=1}^l  \frac{1}{1/4+t_j^2},\ \ \mathrm{as}\ \ n\rightarrow \infty. 
$$
Observe that the set of zeros $r+it_1$, $\dots$, $r+it_l$ was chosen to be arbitrary. Hence,  
$$
\liminf_{n\to \infty} d^2_{n,r} \log n \geq  C,\ \ \mathrm{as}\ \ n\rightarrow \infty, 
$$
where
$$
C=\sum_{\Re(\rho_P)=r}  \frac{1}{|\rho_P-r+1/2|^2}.
$$
$\Box$ 

Finally we comment the result of the theorem for the three possible situations for the location of the zeros of the Dirichlet polynomial.

 1. When $P(s)$ possesses a zero $\rho_P$ in the semi-plane $\Re(s)>r$, then we may write  $\rho_P=r+\delta+it_0$ with $\delta>0$. The function $f(x):=x^{-1/2+\delta-it_0}$ belongs to $L^2 (0,1)$ and by  
 (\ref{Int1f}) we have
 $$
 \int_0^1 \kappa_r \left(\frac{1}{kx}\right)\overline{f(x)}dx=\frac{P(r+\delta+it_0)}{1/2+\delta+it_0}k^{-1/2-\delta-it_0}=0.$$
 Therefore $f$ belongs to the orthogonal complement of the space  $\mbox{span}\{\kappa(1/kx)):k=1,\dots,n\}$ in $L^2(0,1)$  for every $n\in \N$. Thus
 $$
 d^2_{n,r}\geq \frac{ |\int_0^1 \textbf{1}_{(0,1)}(x)\overline{f(x)}dx|^2}{\| f\|^2_{L^2(0,1)}} =\frac{2\delta}{|\rho_P-r+1/2|^2},
 $$
that is, $\liminf_{n\to \infty} d^2_{n,r} \log n = \infty$ and the lower bound proved in Theorem \ref{BE} is superficial. 

2. When all the zeros of $P(s)$ belong to the semi-plane $\Re(s)\leq r$ but there is al least one zero on the line $\Re(s)=r$ the lower bound proved in the latter theorem is useful because Teorema \ref{BP} yields $\lim_{n\to \infty}d^2_{n,r}=0$ and Teorema \ref{BE} provides a lower bound for the speed of convergence of $d^2_{n,r}$. Moreover, employing a slightly modified argument from \cite{BetConFar} we can show that the lower bound represents the true order of decrease of $d_{n,r}$, that is, 
$$
d^2_{n,r}\sim \frac{C}{\log n}, 
$$
for at least some specific Dirichlet polynomials.

3. When all the zeros of $P(s)$ belong to $\Re(s)<r$ the lower bound is trivial because $C=0$.  In this case there are two possibilities: either $\sup_{\rho_P} \{\Re(\rho_P)\}<r$ or  $\sup_{\rho_P} \{\Re(\rho_P)\}=r$. 

In the first one we may apply Lemma \ref{LemLan} with $f(s)=P(s+r-1/2)$, $\sigma_z=\sup_{\rho_P} \{\Re(\rho_P)\}-r+1/2$ and $\epsilon =r-\sup_{\rho_P} \{\Re(\rho_P)\}$ to obtain
$$
\sum_{k=1}^{n}\frac{\mu_{f}(k)}{k^{1/2+it}}=\frac{1}{P(r+it)}+O_{r,\delta}(n^{-\delta})
$$
for each $\delta< r-\sup_{\rho_P} \{\Re(\rho_P)\}$. As we have already mentioned in the proof of Theorem \ref{BE},
$$
d_{n,r}^2=\inf_{{A_n\in \mathcal{D}_n}\atop} \frac{1}{2\pi} \int_{\Re(s)=1/2}\left| \frac{1-P(s+r-1/2)A_n(s)}{s} \right|^2 |ds|.
$$ 
Then choosing $A_n(s)=\sum_{k=1}^{n} \mu_{f}(k) k^{-s}$ we obtain
$$
d_{n,r}^2\leq \frac{1}{2\pi} \int_{\R} \frac{\left| P(r+it)O_{r,\delta}(n^{-\delta}) \right|^2}{1/4+t^2} dt= O_{r,\delta}(n^{-\delta}).
$$
Performing only a slight modification of the arguments it is possible to show that 
$$
d^2_{n,r}\leq o(n^{-\delta})
$$
for each $\delta< r-\sup_{\rho_P} \{\Re(\rho_P)\}$. On the other hand, using some  estimates established in \cite{LubJAT} and reasonings similar to those in the proof of  
Theorem \ref{BE}, it can be shown that 
$$
d^2_{n,r}\neq o(n^{-\delta})
$$
for each  $\delta> 2r-2\sup_{\rho_P} \{\Re(\rho_P)\}$.

The second situation, that is when $\sup_{\rho_P} \{\Re(\rho_P)\}=r$, seems to be more complicated and we expect that a good lower bound for $d_{n,r}$ should depend not only 
on supremum of the real parts of the zeros of $P$ but most probably on their further properties.

\end{document}